\documentclass[11pt,reqno]{amsart}
\usepackage{amssymb,mathrsfs,graphicx}
\usepackage{tikz}
\usepackage{ifthen}
\usepackage{mathtools}
\mathtoolsset{showonlyrefs=true}

\provideboolean{shownotes} 
\setboolean{shownotes}{true} 
\newcommand{\margnote}[1]{
\ifthenelse{\boolean{shownotes}}%
{\marginpar{\raggedright\tiny\texttt{#1}}}%
{}%
}
\newcommand{\hole}[1]{
\ifthenelse{\boolean{shownotes}}%
{\begin{center} \fbox{ \rule {.25cm}{0cm} \rule[-.1cm]{0cm}{.4cm}
\parbox{.85\textwidth}{\begin{center} \texttt{#1}\end{center}} \rule
{.25cm}{0cm}}\end{center}} {} }

\title{Singularity formation for a fluid mechanics model with nonlocal velocity}

\author[Changhui Tan]{Changhui Tan}
\address[Changhui Tan]{\newline Department of Mathematics, \ 
 University of South Carolina, 1523 Greene St., Columbia, SC 29208, USA}
\email{tan@math.sc.edu}



\newtheorem{theorem}{Theorem}[section]
\newtheorem{lemma}[theorem]{Lemma}
\newtheorem{corollary}[theorem]{Corollary}
\newtheorem{proposition}[theorem]{Proposition}
\newtheorem{remark}{Remark}[section]

\newcommand{\R}{\mathbb R}
\newcommand{\T}{\mathbb T}
\newcommand{\pa}{\partial}

\def\u{\mathbf{u}}
\def\x{\mathbf{x}}
\def\y{\mathbf{y}}
\def\rhom{\bar{\rho}}
\def\eps{\varepsilon}

\begin{document}
\allowdisplaybreaks

\date{\today}


\begin{abstract} 
We study a 1D fluid mechanics model with nonlocal velocity. The
equation can be viewed as a fractional porous medium flow, a 1D model
of quasi-geostrophic equation, and also a special case of the
Euler-Alignment system. For strictly positive smooth initial data, global
regularity has been proved in \cite{do2018global}.
We construct a family of non-negative smooth initial data so that 
solution loses $C^1$ regularity. Our result indicates that strict
positivity is a critical condition to ensure global regularity of the system.
We also extend our construction to the corresponding models in
multi-dimensions. 
\end{abstract}

\maketitle \centerline{\date}

\section{Introduction}\label{sec:intro}
We are interested in the following 1D continuity equation
\begin{equation}\label{eq:cont}
\pa_t \rho + \pa_x (\rho u) = 0,
\end{equation}
with a nonlocal velocity field 
\begin{equation}\label{eq:velo}
u=H\Lambda^{\alpha-1}\rho,\quad 0<\alpha<2,
\end{equation}
where $H$ is the Hilbert transform, and $\Lambda^s=(-\Delta)^{s/2}$
denotes the nonlocal fractional
Laplacian operator. The initial density is set to be non-negative
\begin{equation}\label{eq:init}
\rho(x,t)|_{t=0}=\rho_0(x)\geq0.
\end{equation}

 The dynamics of $\rho$ in the system
\eqref{eq:cont}-\eqref{eq:init} can be alternatively written as
\begin{equation}\label{eq:main}
\pa_t\rho+u\pa_x\rho=-\rho\Lambda^\alpha\rho.
\end{equation}
It consists a nonlocal transport term $u\pa_x\rho$, and a
dissipation term $-\rho\Lambda^\alpha\rho$ which is nonlinear and
nonlocal. 

Without the dissipation term, the equation is an active scaler
\begin{equation}\label{eq:activescaler}
\pa_t\rho+u\pa_x\rho=0,
\end{equation}
with the velocity $u$ defined in \eqref{eq:velo}. It arises
as 1D simplified models for 2D surface quasi-geostrophic equations.
For $\alpha=1$, equation \eqref{eq:activescaler} was studied by
C\'ordoba, C\'ordoba and Fontelos \cite{cordoba2005formation}, where a
finite time loss of $C^1$ regularity is shown for some initial data.
Silvestre and Vicol \cite{silvestre2016transport} proved the similar
behavior for $\alpha\in(0,2)$. Both results indicate that the
transport term intends to drive the dynamics into singularity in
finite time.

With the dissipation term, the equation \eqref{eq:main} appears in
many models in fluid mechanics. 
Since the dissipation term has a possible regularizing effect, 
the understanding of the competition between the transport term and
the dissipation term attracts a lot of attentions in recent years.

\subsection*{Fractional porous medium flow} The main system
\eqref{eq:cont}-\eqref{eq:init} can be viewed as a porous medium
equation with fractional potential pressure, where $\rho$ represents
the density of the fluid. It was introduced by
Caffarelli and V\'azquez
\cite{caffarelli2011nonlinear}, where an existence theory for weak
solutions was established, for $\rho_0\in L^1$.
The regularizing effect was discussed in a series of successive works:
\cite{caffarelli2013regularity} for $\alpha\in(0,1)\cup(1,2)$, and
\cite{caffarelli2016regularity} for $\alpha=1$. Their result states
that weak solutions of the system with any $L^1$ initial data
instantly becomes H\"older continuous, and stays in $C^\gamma$ for all
time, with some $\gamma\in(0,1)$. Such regularizing effect is proved in
higher dimensions as well.

For $\alpha=1$, Carrillo, Ferreira and Precioso
\cite{carrillo2012mass} studied the system in the space of probability
measures with bounded second moment. They established a global
wellposedness theory by taking advantage of the gradient flow
structure of the system in 1D.

The system is also related to a model for the motion of the
dislocations in a solid proposed by Biler, Karch and Monneau in
\cite{biler2010nonlinear}.

\subsection*{1D model of quasi-geostrophic equation}
Chae, C\'ordoba, C\'ordoba and Fontelos \cite{chae2005finite} considered
\eqref{eq:cont}-\eqref{eq:init} with $\alpha=1$. They interpreted the
system as a 1D simplified model of 2D
quasi-geostrophic equation in atmospheric science,
where $\rho$ represents the temperature of the air subject to a shift
($\rho=\theta+\kappa$ in their notations).

They studied the system in the periodic domain $\T=[-1/2,1/2]$, and
focused on propagation of regularity with smooth initial data. The result
consists two parts. First, they showed that if $\rho_0>0$, then all $H^3$ initial data
stays in $H^3$ in all time. Second, they proved that the system loses
$C^1$ regularity in finite time, with the initial data chosen as
\begin{equation}\label{eq:initcccf}
\rho_0(x)=1-\cos(2\pi x),\quad x\in\T.
\end{equation}
The main difference between the two types of initial data
is that $\rho_0(x)=0$ is attained in the latter case.
It indicates that the preservation of $C^1$ regularity critically
depends on the strict positivity of the initial data.

In \cite{castro2008global}, Castro and C\'ordoba discussed the blowup
phenomenon for more general initial data without strict positivity.

It is worth noting that $u=H\rho$ when $\alpha=1$. Some properties and
identities of Hilbert transform were crucially used in their proof. So,
the extension of the result to general $\alpha\in(0,2)$ is far from trivial.

\subsection*{Euler-Alignment system}
System \eqref{eq:cont}-\eqref{eq:init} is also related to a
biologically motivated complex interacting system modeling collective
behaviors. The Cucker-Smale model \cite{cucker2007emergent} is an
agent-based model governed by Newton's second law
\begin{equation}\label{eq:ABM}
\dot{x}_i=v_i, \quad
m\dot{v}_i=F_i:=\frac{1}{N}\sum_{j=1}^N\psi(|x_i-x_j|)(v_j-v_i),
\end{equation}
where $(x_i, v_i)_{i=1}^N$ represent the position and velocity of agent
$i$. The force $F_i$ describes the alignment interaction on velocity, where the
influence function $\psi$
characterizes the strength of the velocity alignment between two
agents. Natually, it is a decreasing function of the distance between the agents.

The macroscopic representation of Cucker-Smale model \eqref{eq:ABM},
derived through a kinetic system (see \cite{ha2008particle}), is called
Euler-Alignment system. In 1D, it reads
\begin{align}
&\pa_t \rho + \pa_x (\rho u) = 0,\label{eq:EArho}\\
&\,\,\,\, \pa_t u + u\pa_x u =  \int_{\R} \psi(|x-y|)(u(y,t) - u(x,t))\rho(y,t)dy.\label{eq:EAu}
\end{align}

For the case when $\psi$ is Lipschitz, the system was studied in
\cite{tadmor2014critical,carrillo2016critical}. A critical threshold
phenomenon was discovered: preservation of $C^1$ regularity depends on
the choice of initial data. Subcritical initial data lead to global
regularity, while supercritical initial data lead to fintie time shock
formation.

Another case is when $\psi$ is singular, taking the form
\begin{equation}\label{eq:singularpsi}
\psi(|x|)=\frac{c_\alpha}{|x|^{1+\alpha}},\quad0<\alpha<2,
\end{equation}
with $c_\alpha$ be a positive constant such that
\[\Lambda^\alpha f=c_\alpha\int_\R \frac{f(x)-f(y)}{|x-y|^{1+\alpha}} dy.\]

One interesting feature of such choice of $\psi$ is that,  equation
\eqref{eq:EAu} becomes closely related to the 
Burgers equation with fractional dissipation
\begin{equation}\label{eq:fBurgers}
\pa_tu+u\pa_xu=-\Lambda^\alpha u,
\end{equation}
by enforcing $\rho\equiv1$. Kiselev, Nazarov and Shterenberg
\cite{kiselev2008blow} studied \eqref{eq:fBurgers}: when $0<\alpha<1$,
there exists initial data leading to finite time blow up; when
$\alpha\in[1,2)$, all smooth initial data lead to global regularity.

The Euler-Alignment system \eqref{eq:EArho}-\eqref{eq:EAu} with
singular influence function \eqref{eq:singularpsi} was studied in
\cite{do2018global} in the periodic domain. It was shown that all
smooth initial data $\rho_0>0$ leads to global regularity. In
particular, in the range of $\alpha\in(0,1)$, the behavior of the
solution is very different from the Burgers equation with fractional
dissipation, despite their similarity.
The global regularity result is extended to more general singular
influence function in \cite{kiselev2018global}.
Moreover, it is shown that the $C^1$ norm of the density $\rho$ is
uniformly bounded in all time.
For $\alpha\in[1,2)$, global regularity was independently shown by Shvydkoy
and Tadmor in \cite{shvydkoy2017eulerian} through a different
approach. Their result can also be extended for $\alpha\in(0,1)$ in
\cite{shvydkoy2017eulerian3}.

As discussed in \cite{do2018global}, a useful reformulation of the
Euler-Alignment system for $\rho$ and $G=\pa_xu-\Lambda^\alpha\rho$
has the form
\begin{equation}\label{eq:EA}
\pa_t\rho+\pa_x(\rho u)=0,\quad
\pa_tG+\pa_x(Gu)=0,\quad\pa_xu=\Lambda^\alpha\rho+G.
\end{equation}
In particular, if we pick the initial data such that
$G_0(x)=\pa_xu_0(x)-\Lambda^\alpha\rho_0(x)\equiv0$, then
$G\equiv0$ for all $t>0$, and the dynamics of $\rho$ becomes our main
system \eqref{eq:cont}-\eqref{eq:velo}.

Therefore, the result in \cite{do2018global} implies that for
$\alpha\in(0,2)$, system \eqref{eq:cont}-\eqref{eq:velo} with smooth initial
data $\rho_0>0$ stays smooth in all time. It serves as an extension to the first
part of the result in \cite{chae2005finite} with general $\alpha$.

\subsection*{The main result}
In this paper, we focus on \eqref{eq:cont}-\eqref{eq:velo} with non-negative
initial data $\rho_0$ which is not strictly positive. We construct 
initial data which lead to singularity formations.

\begin{theorem}\label{thm:main}
Consider the system \eqref{eq:cont}-\eqref{eq:init} in the periodic
domain $\T$. There exists a family of smooth initial data $\rho_0$
such that the solution $\rho(\cdot,t)$ is not bounded in $C^1$
uniformly in $t$.
\end{theorem}


Theorem \ref{thm:main} says that the solution will lose $C^1$
regularity as time approaches infinity. Note that this type of singularity
does not happen when $\rho_0>0$ (see \cite{kiselev2018global}).
Hence, the non-vacuum assumption is critical to ensure global
regularity.

Theorem \ref{thm:main} extends the blow up result in
\cite{chae2005finite} to the general case $\alpha\in(0,2)$.
However, it only guarantees singularity formations as time approaches
infinity. Whether the blowup happens in finite time is still an open
problem, which requires future investigations.

As a direct consequence, we have the following result for
Euler-Alignment system.
\begin{corollary}\label{cor:EA}
Consider the initial value problem of Euler-Alignment system
\eqref{eq:EArho}-\eqref{eq:EAu} with singular influence function
$\psi$ defined in \eqref{eq:singularpsi}. There exists smooth
initial data $\rho_0\geq0$ and $u_0$ such that the solution lose
uniform $C^1$ regularity.
\end{corollary}
The choice of initial data could be $\rho_0$ from Theorem
\ref{thm:main}, and $u_0=H\Lambda^{\alpha-1}\rho_0$.

\quad

The rest of the paper is organized as follows. In section
\ref{sec:est}, we show apriori bounds for the system with some
proposed symmetry. In section \ref{sec:enhance}, we obtain an enhanced
estimate on the velocity $u$, which plays an essential rule in
proving the singularity formation.  Theorem \ref{thm:main} is then
proved in section \ref{sec:blowup}.
In section \ref{sec:multi}, we extend the result to systems in multi-dimensional spaces.
Finally, in section
\ref{sec:discuss}, we make some remarks on related topics for further
investigation.





\section{Apriori estimates}\label{sec:est}
In this section, we derive some useful estimates for our main system
\eqref{eq:cont}-\eqref{eq:init}, which will help us to
construct initial data and obtain finite time blow up.

We first propose the following even symmetry condition to $\rho_0$
\begin{equation}\label{eq:h1}
\tag{H1} \rho_0(x)=\rho_0(-x).
\end{equation}

Since we consider periodic data, $\rho_0$ can be determined by
its value in $x\in[0,1/2]$.
We also note that periodicity and even symmetry preserves in time.

\subsection{Maximum principle}
Let us assume the initial data is bounded, satisfying
\begin{equation}\label{eq:h3}
\tag{H2}0\leq\rho_0(x)\leq\rhom,\quad\forall~x\in\T.
\end{equation}
Then, $\rho(\cdot,t)$ satisfies \eqref{eq:h3} for all $t\geq0$, due to
maximum principle.
\begin{proposition}[Maximum principle]\label{prop:max}
Let $\rho$ be a smooth solution of \eqref{eq:cont} with initial data
$\rho_0$ satisfying \eqref{eq:h3}. Then, $\rho(\cdot,t)$ satisfies
\eqref{eq:h3} for all $t\geq0$.
\end{proposition}
\begin{proof}
Suppose $\rho(x,t)\leq\rhom$ does not hold for all $(x,t)$. Then, there
exists $x_0$ and $t_0$ such that 
\[\rho(x_0,t_0)=\rhom,\quad
  \rho(x,t_0)\leq\rhom,~~\forall~x\in\T,\quad\text{and}\quad
\pa_t\rho(x_0,t_0)>0.\]
So the violation first occurs at $x_0$ at time $t_0+$.

Since $\rho(\cdot,t_0)$ attains its maximum at $x_0$, we know
\[\pa_x\rho(x_0,t_0)=0,\quad\text{and}\quad
  \Lambda^\alpha\rho(x_0,t_0)\geq0.\]
Therefore, from \eqref{eq:main} we obtain
\[\pa_t\rho(x_0,t_0)=-u(x_0,t_0)\pa_x\rho(x_0,t_0)-\rho(x_0,t_0)\Lambda^\alpha\rho(x_0,t_0)\leq0.\]
This leads to a contradiction. Therefore, $\rho(x,t)\leq\rhom$ holds for all
$x\in\T$ and $t\geq0$.

Positivity preserving property $\rho(x,t)\geq0$ can be proved similarly.
\end{proof}

\subsection{Conservation of mass}
We denote $m$ as the initial mass
\begin{equation}\label{eq:mass}
  m = \int_\T\rho_0(x)dx.
\end{equation}

Integrating the continuity equation \eqref{eq:cont} in $x$, we get
\[\frac{d}{dt}\int_\T\rho(x,t)dx=-\int_\T\pa_x(\rho(x,t)u(x,t))=0.\]
This implies the conservation of total mass.

Moreover, the mass in any interval is conserved along the
characteristic flow.
\begin{proposition}[Conservation of mass]\label{lem:masscons}
  Let $\rho$ be a strong solution of the continuity equation \eqref{eq:cont}.
  Let $X_1(t), X_2(t)$ be two characteristic paths starting at
  $x_1$ and $x_2$, respectively.
  \[\frac{d}{dt}X_i(t)=u(X_i(t),t),\quad X_i(0)=x_i,\quad i=1,2.\]
  Then,  the mass in the interval $[X_1(t), X_2(t)]$ is conserved
  in time, namely
  \begin{equation}\label{eq:masscons}
    \int_{X(t;x_1)}^{X(t;x_2)}\rho(x,t)dx=\int_{x_1}^{x_2}\rho_0(x)dx,\quad\forall~t\geq0.
  \end{equation}
\end{proposition}
The proof can be found, for instance, in \cite[Lemma
5.1]{tan2019euler}. 

\subsection{Preservation of monotonicity}
We make another assumption on $\rho_0$.
\begin{equation}\label{eq:h2}
\tag{H3} \rho_0(0)=0,\quad \pa_x\rho_0(x)\geq0,~\forall~x\in[0,1/2],
\end{equation}
namely $\rho_0$ is increasing in $[0,1/2]$.

\begin{figure}[h]
\begin{tikzpicture}[scale=1.5]
    \path [draw=black, very thick]
    (0,0) .. controls (.3,.05) .. (.5,1) .. controls (.7, 1.95)
    .. (1.5,2) -- (2.5,2) .. controls (3.3,1.95) .. (3.5,1)
    .. controls (3.7, .05) .. (4,0);

    \coordinate (y) at (0,2.3);
    \coordinate (x) at (4.2,0);
    \draw[<->] (y) node[above] {$\rho_0(x)$} -- (0,0) --  (x) node[right] {$x$};

    \draw (2,0.1) -- (2,0) node[below] {$1/2$};
    \draw (0.1,2) -- (0,2) node[left] {$\rhom$};
    \draw (4,0.1) -- (4,0) node[below] {$1$};

    \draw[dashed] (2,0) -- (2,2.3) node[above] {even};
    \draw[<->] (1.8,2.2) -- (2.2,2.2);
    \draw[dashed] (0,2) -- (1,2);
  \end{tikzpicture}
\caption{The choice of initial data $\rho_0$, satisfying
  \eqref{eq:h1}-\eqref{eq:h2}}\label{fig:init}
\end{figure}
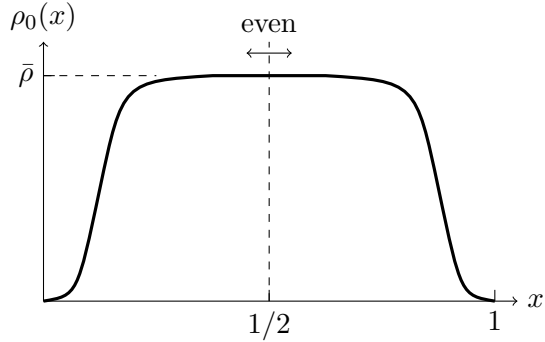

The following proposition shows that such monotonicity is preserved in
time.

\begin{proposition}[Monotonicity]
Assume that $\rho_0$ is smooth and satisfies \eqref{eq:h1}-\eqref{eq:h2}.
Let $\rho$ be a classical solution of
\eqref{eq:cont}-\eqref{eq:init}. Then, $\rho(\cdot,t)$ satisfies
\eqref{eq:h2} for any $t\geq0$.
\end{proposition}
\begin{proof}
Let us denote $\zeta:=\pa_x\rho$, and write down its dynamics by
differentiating \eqref{eq:main} in $x$
\begin{equation}\label{eq:rhox}
\pa_t\zeta = -u\pa_x\zeta-2\zeta\pa_xu-\rho\pa_x^2u=
-u\pa_x\zeta-2\zeta\Lambda^\alpha\rho-\rho\Lambda^\alpha\zeta.
\end{equation}
By periodicity and \eqref{eq:h1}, we know $\zeta(\cdot,t)$ is odd, and so
\[\zeta(0,t)=\zeta(1/2,t)=0.\]
 
Our goal is to prove $\zeta(x,t)\geq0$, for all $x\in[0,1/2]$ and
$t\geq0$.
Assume the argument is false, then there exist
at time $t_0$ and position $x_0\in(0,1/2)$ such that the solution
satisfies
\begin{equation}\label{eq:break}
\zeta(x_0,t_0)=0,\quad\zeta(x,t_0)\geq0,~~\forall~x\in[0,1/2],\quad
\text{and}\quad\pa_t\zeta(x_0,t_0)<0,
\end{equation}
so that the break down first happens at $(x_0,t_0+)$.

Since $\zeta(\cdot,t_0)$ reaches a local minimum at $x_0$, clearly
$\pa_x\zeta(x_0,t_0)=0$. Therefore, the dynamics \eqref{eq:rhox} at
$(x_0,t_0)$ becomes
\begin{equation}\label{eq:rhoxm}
\pa_t\zeta(x_0,t)=-\rho(x_0,t_0)\Lambda^\alpha\zeta(x_0,t_0).
\end{equation}

From Proposition \ref{prop:max}, we know 
$\rho(x_0,t_0)\geq0$. So, we are left to estimate
$\Lambda^\alpha\zeta(x_0,t_0)$.

\begin{align*}
\Lambda^\alpha\zeta(x_0,t_0)
=&c_\alpha\int_\R\frac{\zeta(x_0,t_0)-\zeta(y,t_0)}{|x_0-y|^{1+\alpha}}dy
=-c_\alpha\sum_{l\in\mathbb{Z}}\int_{-1/2}^{1/2}\frac{\zeta(y,t_0)}{|x_0-y-l|^{1+\alpha}}dy\\
=&-c_\alpha\left[\sum_{l\in\mathbb{Z}}\int_0^{1/2}\frac{\zeta(-y,t_0)}{|x_0+y-l|^{1+\alpha}}dy
+\sum_{l\in\mathbb{Z}}\int_0^{1/2}\frac{\zeta(y,t_0)}{|x_0-y-l|^{1+\alpha}}dy\right]\\
=&-c_\alpha\int_0^{1/2}\zeta(y,t_0)\sum_{l\in\mathbb{Z}}
\left(\frac{1}{|x_0-y-l|^{1+\alpha}}-\frac{1}{|x_0+y-l|^{1+\alpha}}\right)dy.
\end{align*}

From \eqref{eq:break} and the following Lemma \ref{lem:pos}, we
conclude that $\Lambda^\alpha\zeta(x_0,t_0)\leq0$ and
hence $\pa_t\zeta(x_0,t_0)\geq0$. This contradicts with the last
inequality in \eqref{eq:break}.
\end{proof}

\begin{lemma}\label{lem:pos}
Suppose $x,y\in[0,1/2]$ and $\alpha>0$. Then
\[\sum_{l\in\mathbb{Z}}
\left(\frac{1}{|x-y-l|^{1+\alpha}}-\frac{1}{|x+y-l|^{1+\alpha}}\right)\geq0.\]
\end{lemma}
\begin{proof}
We first consider the case when $y\leq x$. The sum can be rewritten as
\[
\sum_{l\geq1} \left[\left(\frac{1}{(l-1+x-y)^{1+\alpha}}-\frac{1}{(l-x-y)^{1+\alpha}}\right)-
\left(\frac{1}{(l-1+x+y)^{1+\alpha}}-\frac{1}{(l-x+y)^{1+\alpha}}\right)\right].
\]
Define
\[H_l(z)=\frac{1}{(l-1+x-z)^{1+\alpha}}-\frac{1}{(l-x-z)^{1+\alpha}}.\]
Then, the sum can be represented as
\[\sum_{l\geq1} (H_l(y)-H_l(-y)).\]
Since we have
\[H_l'(z)=(1+\alpha)\left[\frac{1}{(l-1+x-z)^{2+\alpha}}-\frac{1}{(l-x-z)^{2+\alpha}}\right]\geq0,\quad\forall~z\in[-1/2,1/2],\]
we get $H_l(y)-H_l(-y)\geq0$ for any $y\in[0,x]$. It implies that the sum
is non-negative.

The case when $y>x$ can be treated in the same way.
\end{proof}

\subsection{An estimate on velocity}
The velocity $u$ defined in \eqref{eq:velo} can be expressed in the integral form as follows:
\begin{equation}\label{eq:u}
u(x,t)=c_\alpha\int_\R\frac{\rho(y,t)-\rho(x,t)}{sgn(x-y)|x-y|^\alpha}dy.
\end{equation}

Fix $x\in[0,1/2]$ and $t\geq0$. We decompose the integrand and use
\eqref{eq:h1} to get
\begin{align*}
\frac{1}{c_\alpha}u(x,t)=&\int_0^\infty\frac{\rho(y,t)-\rho(x,t)}{|x+y|^\alpha}dy
+\int_0^x\frac{\rho(y,t)-\rho(x,t)}{|x-y|^\alpha}dy
-\int_x^\infty\frac{\rho(y,t)-\rho(x,t)}{|x-y|^\alpha}dy\\
=&\int_0^x(\rho(y,t)-\rho(x,t))
\left(\frac{1}{(x+y)^\alpha}+\frac{1}{(x-y)^\alpha}\right)dy\\
&+\int_x^\infty(\rho(y,t)-\rho(x,t))
  \left(\frac{1}{(x+y)^\alpha}-\frac{1}{(y-x)^\alpha}\right)dy=:I+II.\end{align*}
Due to monotonicity condition of $\rho(\cdot,t)$ \eqref{eq:h2}, we know that the first term $I\leq0$.
For the second term $II$, observe that
\[\frac{1}{(x+y)^\alpha}-\frac{1}{(y-x)^\alpha}<0,\quad\forall~y>x>0.\]
So, the integral in $II$ can be decompose into two parts:
\[
\int_x^\infty=\sum_{l=0}^\infty\int_{l+x}^{l+1-x}+\sum_{l=1}^\infty\int_{l-x}^{l+x}.
\]
Again, condition \eqref{eq:h2} implies that  for the first part
$\rho(y,t)-\rho(x,t)\geq0$, and for the second part $\rho(y,t)-\rho(x,t)\leq0$.
Let us denote $II=II_1+II_2$ where $II_1$ and $II_2$ represents the
corresponding integrals. Then, $II_1\leq0$ and $II_2\geq0$.

The next lemma shows $I+II_2\leq 0$, at least when $x$ is sufficiently
small.
\begin{lemma}\label{lem:dominate}
There exists a $\delta=\delta(\alpha)>0$, such that for all $x\in[0,\delta]$, $I+II_2\leq0$.
\end{lemma}
\begin{proof}
Let us first write
\[II_2=\int_{-x}^x(\rho(x,t)-\rho(y,t))\sum_{l=1}^\infty
  \left(\frac{1}{(y+l-x)^\alpha}-\frac{1}{(y+l+x)^\alpha}\right)dy.\]
Using mean value theorem, we have for $y\in(-x,x)$,
\[\frac{1}{(y+l-x)^\alpha}-\frac{1}{(y+l+x)^\alpha}\leq
\alpha(l-2x)^{-1-\alpha}\cdot(2x).
\]
Therefore,
\[\sum_{l=1}^\infty\frac{1}{(y+l-x)^\alpha}-\frac{1}{(y+l+x)^\alpha}\leq
2\alpha
x\left[(1-2x)^{-1-\alpha}+\int_1^\infty(z-2x)^{-1-\alpha}dz\right]
\leq Cx.
\]
For $x\leq1/4$, the last inequality holds with the choice of
$C=2^{\alpha+1}(1+2\alpha)$.

Now, let us put together $I$ and $II_2$.
\begin{align*}
I+II_2=&\int_0^x(\rho(x,t)-\rho(y,t))\left[-\frac{1}{(x+y)^\alpha}-\frac{1}{(x-y)^\alpha}\right.\\
&+\left.\sum_{l=1}^\infty\left(\frac{1}{(y+l-x)^\alpha}-\frac{1}{(y+l+x)^\alpha}
+\frac{1}{(-y+l-x)^\alpha}-\frac{1}{(-y+l+x)^\alpha}\right)\right]dy\\
\leq&
\int_0^x(\rho(x,t)-\rho(y,t))\left[-\frac{1}{(x-y)^\alpha}+0+2Cx\right]dy\\
\leq& (-x^{-\alpha}+2Cx)\int_0^x(\rho(x,t)-\rho(y,t))dy.
\end{align*}
We pick a small enough $\delta$ as follows
\begin{equation}\label{eq:delta}
\delta=\min\left\{\frac{1}{4},\left(\frac{1}{3C}\right)^{\frac{1}{1+\alpha}}\right\},
\end{equation}
Then, for any $x\in(0,\delta]$, we have $-x^{-\alpha}+2Cx\leq-Cx<0$. 

Also, the monotonicity condition \eqref{eq:h2} implies that
\[\int_0^x(\rho(x,t)-\rho(y,t))dy\geq0.\]

Therefore, conclude that $I+II_2\leq0$ for all $x\in[0,\delta]$.
\end{proof}

Lemma \ref{lem:dominate} directly implies the following estimate on
$u$.
\begin{theorem}\label{thm:roughu}
Let $\rho$ be a classical solution of \eqref{eq:cont}-\eqref{eq:init},
with periodic initial data $\rho_0$ satisfying
\eqref{eq:h1}-\eqref{eq:h2}. Let $\delta$ be defined as
\eqref{eq:delta}. Then, the velocity 
\[u(x,t)\leq 0,\quad\forall~x\in[0,\delta],~~t\geq0.\]
\end{theorem}

One may remove the smallness assumption on $x$ in Theorem
\ref{thm:roughu} by a more careful estimate on $II_2$. For our
purpose, it is enough to consider small $x$.

\section{An enhanced estimate on velocity}\label{sec:enhance}
In order to show singularity formations, we need a stronger estimate
on the velocity.
Recall
\[u(x,t)=(I+II_2)+II_1.\]
Lemma \ref{lem:dominate} ensures $I+II_2\leq0$.
The estimate $II_1\leq0$ simply follows for \eqref{eq:h2}.

We aim to improve our estimate on
\[II_1=-\sum_{l=0}^\infty\int_{l+x}^{l+1-x}
(\rho(y,t)-\rho(x,t))
\left(\frac{1}{(y-x)^\alpha}-\frac{1}{(y+x)^\alpha}\right)dy.\]

An easy observation is that, if $\rho(x,t)=\rhom$, then
$II_1=0$. In this case, it is not possible to get any improvement.
Therefore, we obtain an enhanced estimate when $\rho(x,t)$ is small.

\begin{theorem}\label{thm:improveu}
Let $\rho$ be a classical solution of \eqref{eq:cont}-\eqref{eq:init},
with periodic initial data $\rho_0$ satisfying
\eqref{eq:h1}-\eqref{eq:h2}, Let $\delta$ be defined as \eqref{eq:delta}.
Then, there exists a positive constant $A=A(\alpha, m, \rhom)>0$,
for any $(x,t)$ satisfying $x\in[0,\delta]$ and
\begin{equation}\label{eq:smallrho}
  \rho(x,t)\leq\frac{m}{2},
\end{equation}
the velocity
\begin{equation}\label{eq:improveu}
u(x,t)\leq -Ax.
\end{equation}
\end{theorem}

Let us explain the main idea of the proof. 
We focus on a better bound on $[x,1/2]$, and use the rough
bound by zero for the rest of the integrand.
\[II_1\leq-\int_x^{1/2}(\rho(y,t)-\rho(x,t))
  \left(\frac{1}{(y-x)^\alpha}-\frac{1}{(y+x)^\alpha}\right)dy.
\]
Denote the term that we concern by $III$.
\[III=\int_x^{1/2}(\rho(y,t)-\rho(x,t))h(x,y)dy,\quad
h(x,y)=\frac{1}{(y-x)^\alpha}-\frac{1}{(y+x)^\alpha}.\]

To obtain a lower bound on $III$, we need several observations. First,
for a fixed $x\in[0,\delta]$, $h(x,y)\geq0$ for any
$y\in(x,1/2]$. Moreover,
\begin{equation}\label{eq:hdec}\pa_yh(x,y)=-\alpha\left[\frac{1}{(y-x)^{\alpha+1}}-\frac{1}{(y+x)^{\alpha+1}}\right]\leq0.
\end{equation}
Next, we apply \eqref{eq:h3} \eqref{eq:h2}, and get
\begin{equation}\label{eq:rhob}
0\stackrel{\eqref{eq:h2}}{\leq}\rho(y,t)-\rho(x,t)\stackrel{\eqref{eq:h3}}{\leq} \rhom-\rho(x,t),\quad\forall~y\in(x,1/2].
\end{equation}
Moreover, the assumption \eqref{eq:smallrho} implies
\begin{equation}\label{eq:masslb}
\int_x^{1/2}\left(\rho(y,t)-\rho(x,t)\right)dy\stackrel{\eqref{eq:h2}}{\geq}
\int_0^{1/2}\left(\rho(y,t)-\rho(x,t)\right)dy=\frac{m}{2}-\frac{\rho(x,t)}{2}
\stackrel{\eqref{eq:smallrho}}{\geq}\frac{m}{4}.
\end{equation}

The following lemma is helpful to get a positive lower bound of $III$.

\begin{lemma}\label{lem:min}
Let $f$ be a positive decreasing function on $[a,b]$. $\lambda$ and $M$ are positive
constant such that $\lambda<M(b-a)$. Then,
\[\min_\omega\left\{\int_a^b\omega(x)f(x)dx ~\left|~ 0\leq\omega(x)\leq M,
\int_a^b\omega(x)dx\geq \lambda\right.\right\}=M\int^b_{b-\frac{\lambda}{M}}f(x)dx.\]
The minimum is attained at 
\[\omega_{\min}(x)=\begin{cases}0&a\leq x<b-\frac{\lambda}{M}\\ M&b-\frac{\lambda}{M}\leq x\leq b\end{cases}.\]
\end{lemma}
\begin{proof}
First, it is easy to check $\omega_{\min}$ satisfies
\begin{equation}\label{eq:omegares}
 0\leq\omega(x)\leq M,\quad\int_a^b\omega(x)dx\geq \lambda.
\end{equation}
We will prove that for any $\omega$ which satisfies
\eqref{eq:omegares}, $\int_a^b(\omega(x)-\omega_{\min}(x))f(x)dx\geq0$.
Compute
\[\int_a^b(\omega(x)-\omega_{\min}(x))f(x)dx
=\int_a^{b-\frac{\lambda}{M}}\omega(x)f(x)dx+\int_{b-\frac{\lambda}{M}}^b(\omega(x)-M)f(x)dx.\]
From the first condition in \eqref{eq:omegares}, we know
$\omega(x)\geq0$ and $\omega(x)-M\leq0$. Together with the assumption
that $f$ is positive and decreasing, we obtain
\begin{align*}\int_a^b(\omega(x)-\omega_{\min}(x))f(x)dx
\geq&~
      f(b-\frac{\lambda}{M})\int_a^{b-\frac{\lambda}{M}}\omega(x)dx+f(b-\frac{\lambda}{M})\int_{b-\frac{\lambda}{M}}^b(\omega(x)-M)dx\\
=&~f(b-\frac{\lambda}{M})\left[\int_a^b\omega(x)dx-M\cdot\frac{\lambda}{M}\right]
\geq f(b-\frac{\lambda}{M})(\lambda-\lambda)=0.
\end{align*}
Hence, we conclude
\[\min_{\omega~ \text{satisfies}~
    \eqref{eq:omegares}}\int_a^b\omega(x)f(x)dx=
\int_a^b\omega_{\min}(x)f(x)dx=M\int^b_{b-\frac{\lambda}{M}}f(x)dx.\]
\end{proof}

Putting together \eqref{eq:hdec}, \eqref{eq:rhob} and
\eqref{eq:masslb}, 
we can apply Lemma \ref{lem:min} with 
\[f(y)=h(x,y),~~ \omega(y)=\rho(y,t)-\rho(x,t), ~~
  \lambda=\frac{m}{4}, ~~ M=1-\rho(x,t), ~~ a=x, ~~ b=\frac{1}{2}.\]
Then, 
\[III\geq(1-\rho(x,t))\int_{\frac{1}{2}-\frac{m}{4(\rhom-\rho(x,t))}}^{\frac{1}{2}}h(x,y)dy
  \geq(1-\rho(x,t))\int_{\frac{1}{2}-\frac{m}{4\rhom}}^{\frac{1}{2}}h(x,y)dy.\]

Using the mean value theorem, we have
\[h(x,y)\geq\frac{\alpha}{(y+x)^{1+\alpha}}\cdot(2x)\geq2\alpha x.\]
Finally, we obtain
\[III\geq\frac{1}{2}\cdot\frac{m}{4\rhom}\cdot(2\alpha x)=\frac{\alpha
  m}{4\rhom}x,\]
and therefore
\[II_1\leq-\frac{\alpha m}{4\rhom}x.\]
We end up with the improved estimate \eqref{eq:improveu} with
$A=\frac{\alpha m}{4\rhom}$.

\section{Singularity formation}\label{sec:blowup}
In this section, we prove Theorem \ref{thm:main}: for any smooth
initial data satisfying \eqref{eq:h1}-\eqref{eq:h2}, the solution
loses uniform $C^1$ regularity.

We will argue by contradiction. Suppose the solution is uniformly
$C^1$ in all time, then there exists $\eps>0$ such that
\begin{equation}\label{eq:usmall}
  u(\eps, t)\leq \frac{m}{2}, \quad\forall~t\geq0.
\end{equation}
Without loss of generosity, we assume $\eps\leq\delta$. In fact, if
$\eps>\delta$, $u(\delta,t)\leq u(\eps,t)\leq\frac{m}{2}$ by
\eqref{eq:h2}. We can then take $\eps=\delta$.

Let us denote $X(t; x)$ be the characteristic path initiated at $x$, satisfying
\[\frac{d}{dt}X(t; x)=u(X(t; x), t),\quad X(0; x)=x.\]
By symmetry, we know $u(0,t)=0$ and hence $X(t; 0)=0$ for all $t\geq0$.

Define $m(x,t)$ be the mass in the interval $[0,x]$ at time $t$:
\[m(x,t):=\int_0^x\rho(x,t)dx.\]
We apply Proposition \ref{lem:masscons} and get
\begin{equation}\label{eq:localmasscon}
  m(X(t;x),t)=m(x,0).
\end{equation}

Let $x_0=\inf\{x\geq0~:~\rho_0(x)>0\}$. By \eqref{eq:h1} and
\eqref{eq:h2}, we have
\[\text{supp}(\rho_0)=(x_0, 1-x_0).\]

We shall proceed with two cases.
\subsection*{Case 1: $x_0<\eps$}
By the definition of $x_0$, we know $\rho_0(\eps)>0$. Moreover, $m(\eps, 0)>0$.

By Theorem \ref{thm:roughu}, we know $X(t; \eps)\leq\eps$ for any $t\geq0$. Then, the
assumption \eqref{eq:usmall} ensures that $\rho(X(t;
\eps))\leq\frac{m}{2}$ in all time. This allows us to use the enhanced
estimate, Theorem \ref{thm:improveu}, and get
\[u(X(t;\eps),t)\leq-AX(t;\eps),\]
where $A>0$ does not depend on $\eps$ or $t$.

Then, we can integrate along the characteristic path, and get
\[X(t;\eps)\leq \eps e^{-At}.\]

A simple estimate yields
\[m(X(t;\eps),t)=\int_0^{X(t;\eps)}\rho(x,t)dx\stackrel{\eqref{eq:h2}}{\leq}
X(t;\eps)\rho(X(t;\eps),t)\leq\frac{\eps m}{2}e^{-At}.\]

This contradicts with the mass conservation \eqref{eq:localmasscon} if
we pick $t$ large enough, more precisely,
\begin{equation}\label{eq:tlarge}
  t>\frac{1}{A}\log\frac{\eps m}{2m(\eps,0)}.
\end{equation}

\begin{remark}
  If $\rho_0(x)=0$ only at a single point $x=0$, then $x_0=0$. No
  matter what $\eps$ is, we are always under this case. Therefore, we
  have already shown the singularity formation. Note that the initial
  data \eqref{eq:initcccf} lie into this category.
\end{remark}

\subsection*{Case 2: $x_0\geq\eps$}
If $x_0>0$, namely $\rho_0(x)=0$ in an interval $[-x_0,x_0]$, it is
possible that $\eps\leq x_0$.
Then, $m(\eps,0)=0$. Consequently, the right hand
side of \eqref{eq:tlarge} is not bounded any more.

To obtain a contradiction, we first examine the characteristic path
starting at $x_0$. Since $\rho_0(x_0)=0$, it is easy to see that
$\rho_0(X(t;x_0),t)=0$ at any time. We can apply the enhanced estimate
\eqref{eq:improveu} at $(X(t;x_0),t)$, and obtain
\[X(t;x_0)\leq x_0e^{-At}.\]
Then, there exists a finite time $T_*$ such that
$X(t;x_0)\leq\eps$. For instance, one can take
\[T_*=\frac{1}{A}\log\frac{x_0}{\eps}.\]

Now, we consider the characteristic path that goes through the point
$(\eps, T_*+1)$. If the flow is smooth, we can track back and find a
unique point $x_*$ such that $\eps=X(T_*+1; x_*)$.

Moreover, as $X(T_*+1; x_0)<\eps$, we have $x_0<x_*$. By the
definition of $x_0$, we know $\rho_0(x_*)>0$ and hence $m(x_*,0)>0$.

Now, we can repeat the argument in case 1 along $X(t; x_*)$.
First, apply the enhanced estimate \eqref{eq:improveu} at $(X(t;
x_*),t)$ for $t\geq T_*+1$ and get
\[X(t; x_*)\leq \eps e^{-A(t-(T_*+1))},\quad\forall~t\geq T_*+1.\]
Next, we estimate the mass
\[m(X(t;x_*),t)\leq
  X(t;x_*)\rho(X(t;x_*),t)\leq\frac{\eps m}{2}e^{-A(t-(T_*+1))},
  \quad\forall~t\geq T_*+1.\]
Finally, take $t$ large enough
\[  t>\frac{1}{A}\log\frac{\eps m}{2m(x_*,0)}+(T_*+1).\]
Then, $m(X(t;x_*),t)<m(x_*,0)$, which contradicts with the mass
conservation \eqref{eq:localmasscon}.

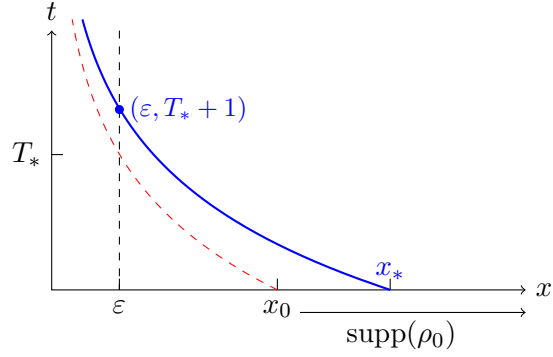
\begin{figure}[h]
\begin{tikzpicture}[scale=1.5]

    \coordinate (y) at (0,2.3);
    \coordinate (x) at (4.2,0);
    \draw[<->] (y) node[above] {$t$} -- (0,0) --  (x) node[right] {$x$};

    \draw (.6,0.1) -- (.6,0) node[below] {$\eps$};
    \draw (2,0.1) -- (2,0) node[below] {$x_0$};
    \draw (0.1,1.2) -- (0,1.2) node[left] {$T_*$};
    \draw (3,0.1) -- (3,0) node[above, blue] {$x_*$};
    \draw [->] (2.2,-.2) -- (4.2, -.2);
    \node at (3.1, -.4) {$\text{supp}(\rho_0)$};

    \draw[dashed] (.6,2.3) -- (.6,0);

     \draw[dashed, domain=0:2.4,smooth,variable=\y,red]  plot ({2*exp(-\y)},{\y});
     \draw[thick, domain=0:2.4,smooth,variable=\y,blue]  plot ({3*exp(-\y)},{\y});

     \filldraw[blue] (.6,1.6) circle (1pt) node[right] {\small$(\eps, T_*+1)$};
  \end{tikzpicture}
\caption{The characteristic path that leads to a contradiction}\label{fig:point}
\end{figure}

\section{Extension to systems in multi-dimensions}\label{sec:multi}
In this section, we extend our main result to systems in higher
dimensions. The main idea is to consider $\rho_0(\x)=\rho_0(x_1)$ and
reduce the system to 1D so that our construction can be used.

\subsection{Fractional porous medium flow}\label{sec:multipmf}
Let us recall the fractional porous medium flow in multi-dimension
\begin{equation}\label{eq:multi}
\pa_t\rho+\nabla\cdot(\rho\u)=0,\quad
\u=\nabla\Lambda^{\alpha-2}\rho, 
\end{equation}
with $\x=(x_1,\cdots,x_n)\in\T^n$ and $0<\alpha<2$.

Fix any time $t$ and drop the time dependence for simplicity.
Assume $\rho(\x)=\rho(x_1)$, namely $\rho$ is a constant in
$(x_2,\cdots,x_n)$ variables.
We calculate the velocity field $u$, starting with
\[\Lambda^{\alpha-2}\rho
=c_{n,\alpha}\int_{\R^n}\rho(\x-\y)\frac{1}{|\y|^{n+\alpha-2}}d\y.\]
Then, we obtain $u$ by taking the gradient of the potential
\[u_i(\x)=\pa_{x_i}\Lambda^{\alpha-2}\rho
=c_{n,\alpha}\int_{\R^n}(\rho(x_1-y_1)-\rho(x_1))\frac{y_i}{|\y|^{n+\alpha}}d\y.\]

For $i=2,\cdots,n$, we have
\begin{equation}\label{eq:u2}
u_i(\x)=c_{n,\alpha}\int_\R(\rho(x_1-y_1)-\rho(x_1))
\left[\int_{\R^{n-1}}\frac{y_i}{|\y|^{n+\alpha}}dy_2\cdots
  dy_n\right]dy_1=0.
\end{equation}
The last eqaulity is due to oddness of the inside integral with
respect to $y_i$.

For $i=1$,
\[u_1(\x)=c_{n,\alpha}\int_\R(\rho(x_1-y_1)-\rho(x_1))y_1
\left[\int_{\R^{n-1}}\frac{1}{|\y|^{n+\alpha}}dy_2\cdots
  dy_n\right]dy_1.\]
Compute the integral inside,
\begin{align*}
\int_{\R^{n-1}}\frac{1}{|\y|^{n+\alpha}}&dy_2\cdots dy_n
=\int_{\R^{n-1}}\left(y_1^2+y_2^2+\cdots+y_n^2\right)^{-\frac{n+\alpha}{2}}dy_2\cdots
  dy_n\\
=&~|y_1|^{-(n+\alpha)}\int_{\R^{n-1}}\left(1+y_2^2+\cdots+y_n^2\right)^{-\frac{n+\alpha}{2}}|y_1|^{n-1}dy_2\cdots
   dy_n\\
=&~|y_1|^{-1-\alpha}\omega_{n-1}\int_0^\infty(1+r^2)^{-\frac{n+\alpha}{2}}r^{n-2}dr
=c'_{n,\alpha}|y_1|^{-1-\alpha}.
\end{align*}
Here, $\omega_{n}$ denotes the area of the unit sphere in $n$
dimension. The constant $c'_{n,\alpha}$ is clearly positive, finite,
and only depend on $n$ and $\alpha$.

Then, we obtain
\begin{equation}\label{eq:u1}
u_1(\x)=c_{n,\alpha}c'_{n,\alpha}\int_\R\frac{\rho(x_1-y_1)-\rho(x_1)}{sgn(y_1)|y_1|^{\alpha}}dy_1.
\end{equation}
So, $u_1(\x)=u_1(x_1)$ is also a constant in
$(x_2,\cdots,x_n)$. Moreover, as a function of $x_1$, the expression
of $u_1$ is the same as \eqref{eq:u}, except the constant $c_\alpha$
might be different.

From \eqref{eq:u2} and \eqref{eq:u1}, we have
\[\nabla\cdot(\rho(\x) \u(\x))=\pa_{x_1}(\rho(x_1)u_1(x_1)).\]
This implies if $\rho_0(\x)=\rho_0(x_1)$, then
$\rho(\x,t)=\rho(x_1,t)$. Moreover, $(\rho,u_1)$ as functions of
$x_1$, will be the solution of the 1D system
\eqref{eq:cont}-\eqref{eq:init}.
Hence, Theorem \ref{thm:main} can be extended to multi-dimension, with
the choice of initial data $\rho_0(\x)=\rho_0(x_1)$, where $\rho_0$ as
a function of $x_1$ is chosen the same way as in the 1D case. The
different constant in \eqref{eq:u1} mentioned above will only affect
the choice of $\delta$ throughout the proof.

We summarize the discussion to the following theorem.
\begin{theorem}\label{thm:multid}
Consider the initial value problem of system \eqref{eq:multi} in the periodic domain
$\T^n$. There exists a family of smooth initial data $\rho_0$ such
that the solution loses uniform $C^1$ regularity.
\end{theorem}

\subsection{Fractional Euler-Alignment system}
The multi-dimensional Euler-Alignment system with singular
influence function takes the form
\begin{equation}\label{eq:EAmulti}
\pa_t\rho+\nabla\cdot(\rho \u)=0,\quad
\pa_t\u+\u\cdot\nabla
\u=c_{n,\alpha}\int_{\R^n}\frac{\u(y,t)-\u(x,t)}{|y-x|^{n+\alpha}}\rho(y,t)dy.
\end{equation}

Let $G=\nabla\cdot \u-\Lambda^\alpha\rho$. Then, the dynamics of $G$
reads
\[\pa_tG+\nabla\cdot(G \u)=\text{tr}(\nabla \u^{\otimes2})-(\nabla\cdot
 \u)^2.\]

Note that In the 1D case, the right hand side becomes $(\pa_xu)^2-(\pa_xu)^2=0$. Then,
the dynamics becomes \eqref{eq:EA}, and as a special case of
$G\equiv0$, we reach our system \eqref{eq:cont}-\eqref{eq:velo}.

However, the right hand side is not necessarily zero in higher
dimensions. This quantity is known as \emph{spectral gap}. In particular, it
destroys the maximum principle on $G$, and hence $G_0\equiv0$ does not
imply $G(\cdot,t)\equiv0$.

Therefore, fractional porous median flow \eqref{eq:multi} is not a
special case of the Euler-Alignment system, except in 1D. The global
regularity on \eqref{eq:EAmulti} for $\rho_0>0$ is an open
problem. The main difficulty is the lack of apriori control of the
spectral gap.

To construct $\rho_0\geq0$ which leads to singularity formations, we can
avoid the difficulty by select a special family of initial data such
that the spectral gap is zero in all time.

The choices of $(\rho_0, \u_0)$ is the same as Section
\ref{sec:multipmf}:
\[\rho_0(\x)=\rho_0(x_1),\quad
(u_0)_1(\x)=(u_0)_1(x_1),\quad (u_0)_i(\x)=0,~~\forall~i=2,\cdots,n.\]

By the same argument, we know such structure preserves in
time. So,
\[\text{tr}(\nabla \u^{\otimes2})-(\nabla\cdot\u)^2=(\pa_{x_1}u_1)^2-(\pa_{x_1}u_1)^2=0.\]

Therefore, we pick $\rho_0$ the same as in Theorem \ref{thm:multid},
and $\u_0=\nabla\Lambda^{\alpha-2}\rho_0$. The solution will form singularities
the same way as \eqref{eq:multi}.

\begin{corollary}\label{cor:multid}
Consider the initial value problem of system \eqref{eq:EAmulti} in the periodic domain
$\T^n$. There exists a family of smooth initial data $(\rho_0,\u_0)$ such
that the solution loses uniform $C^1$ regularity.
\end{corollary}

\section{Further discussions}\label{sec:discuss}
Theorem \ref{thm:main} shows singularity formations for equations
\eqref{eq:cont}-\eqref{eq:init}. However, it does not specify whether
the blowup happens in finite time or when time approaches infinity.

For the special case with $\alpha=1$ and initial data
\eqref{eq:initcccf}, a finite time blowup was shown in
\cite{chae2005finite}. Therefore, a reasonable conjecture would be,
the singularity formations happen at a finite time.

The proof of the conjecture will require a stronger estimate on the
velocity field
\[u(x,t)\leq-Cx^\gamma,\]
with $\gamma<1$. This will ensure the characteristic paths intersect in
finite time, causing a blowup.
To obtain the strong inequality, a delicate estimate to the singular
integral near the singularity is required.
We will leave it for future investigations.

\bigskip\noindent
\textbf{Acknowledgments.} 
This work is supported by NSF grant DMS 1853001.
The author would like to thank Tam Do, Alexander Kiselev and Xiaoqian
Xu for valuable discussions.
The author also thank the referees for valubale suggestions.

\bibliographystyle{plain}
\bibliography{singular}

\begin{thebibliography}{10}

\bibitem{biler2010nonlinear}
Piotr Biler, Grzegorz Karch, and R{\'e}gis Monneau.
\newblock Nonlinear diffusion of dislocation density and self-similar
  solutions.
\newblock {\em Communications in Mathematical Physics}, 294(1):145--168, 2010.

\bibitem{caffarelli2013regularity}
Luis~A Caffarelli, Fernando Soria, and Juan~L V{\'a}zquez.
\newblock Regularity of solutions of the fractional porous medium flow.
\newblock {\em Journal of the European Mathematical Society}, 15(5):1701--1746,
  2013.

\bibitem{caffarelli2011nonlinear}
Luis~A Caffarelli and Juan~L V{\'a}zquez.
\newblock Nonlinear porous medium flow with fractional potential pressure.
\newblock {\em Archive for Rational Mechanics and Analysis}, 202(537--565),
  2011.

\bibitem{caffarelli2016regularity}
Luis~A Caffarelli and Juan~L V{\'a}zquez.
\newblock Regularity of solutions of the fractional porous medium flow with
  exponent 1/2.
\newblock {\em St. Petersburg Mathematical Journal}, 27(3):437--460, 2016.

\bibitem{carrillo2016critical}
Jos{\'e}~A Carrillo, Young-Pil Choi, Eitan Tadmor, and Changhui Tan.
\newblock Critical thresholds in {1D} {E}uler equations with nonlocal forces.
\newblock {\em Mathematical Models and Methods in Applied Sciences},
  26(1):185--206, 2016.

\bibitem{carrillo2012mass}
Jos{\'e}~A Carrillo, Lucas~CF Ferreira, and Juliana~C Precioso.
\newblock A mass-transportation approach to a one dimensional fluid mechanics
  model with nonlocal velocity.
\newblock {\em Advances in Mathematics}, 231(1):306--327, 2012.

\bibitem{castro2008global}
A~Castro and D~C{\'o}rdoba.
\newblock Global existence, singularities and ill-posedness for a nonlocal
  flux.
\newblock {\em Advances in Mathematics}, 219(6):1916--1936, 2008.

\bibitem{chae2005finite}
Dongho Chae, Antonio C{\'o}rdoba, Diego C{\'o}rdoba, and Marco~A Fontelos.
\newblock Finite time singularities in a {1D} model of the quasi-geostrophic
  equation.
\newblock {\em Advances in Mathematics}, 194(1):203--223, 2005.

\bibitem{cordoba2005formation}
Antonio C{\'o}rdoba, Diego C{\'o}rdoba, and Marco~A Fontelos.
\newblock Formation of singularities for a transport equation with nonlocal
  velocity.
\newblock {\em Annals of mathematics}, pages 1377--1389, 2005.

\bibitem{cucker2007emergent}
Felipe Cucker and Steve Smale.
\newblock Emergent behavior in flocks.
\newblock {\em Automatic Control, IEEE Transactions on}, 52(5):852--862, 2007.

\bibitem{do2018global}
Tam Do, Alexander Kiselev, Lenya Ryzhik, and Changhui Tan.
\newblock Global regularity for the fractional {E}uler alignment system.
\newblock {\em Archive for Rational Mechanics and Analysis}, 228(1):1--37,
  2018.

\bibitem{ha2008particle}
Seung-Yeal Ha and Eitan Tadmor.
\newblock From particle to kinetic and hydrodynamic descriptions of flocking.
\newblock {\em Kinetic and Related Models}, 1(3):415--435, 2008.

\bibitem{kiselev2008blow}
Alexander Kiselev, Fedor Nazarov, and Roman Shterenberg.
\newblock Blow up and regularity for fractal burgers equation.
\newblock {\em Dynamics of PDE}, 5(3):211--240, 2008.

\bibitem{kiselev2018global}
Alexander Kiselev and Changhui Tan.
\newblock Global regularity for {1D} {E}ulerian dynamics with singular
  interaction forces.
\newblock {\em SIAM Journal on Mathematical Analysis}, 50(6):6208--6229, 2018.

\bibitem{shvydkoy2017eulerian}
Roman Shvydkoy and Eitan Tadmor.
\newblock {E}ulerian dynamics with a commutator forcing.
\newblock {\em Transactions of Mathematics and Its Applications}, 1(1), 2017.

\bibitem{shvydkoy2017eulerian3}
Roman Shvydkoy and Eitan Tadmor.
\newblock {E}ulerian dynamics with a commutator forcing iii. fractional
  diffusion of order $0<\alpha<1$.
\newblock {\em Physica D: Nonlinear Phenomena}, 2017.

\bibitem{silvestre2016transport}
Luis Silvestre and Vlad Vicol.
\newblock On a transport equation with nonlocal drift.
\newblock {\em Transactions of the American Mathematical Society},
  368(9):6159--6188, 2016.

\bibitem{tadmor2014critical}
Eitan Tadmor and Changhui Tan.
\newblock Critical thresholds in flocking hydrodynamics with non-local
  alignment.
\newblock {\em Philosophical Transactions of the Royal Society of London A:
  Mathematical, Physical and Engineering Sciences}, 372(2028):20130401, 2014.

\bibitem{tan2019euler}
Changhui Tan.
\newblock On the euler-alignment system with weakly singular communication
  weights.
\newblock {\em arXiv preprint arXiv:1901.02582}, 2019.

\end{thebibliography}

\end{document}